\documentclass[reqno,a4paper,12pt]{amsart}

\usepackage[hyperindex,breaklinks]{hyperref}
\usepackage{amssymb}

\newtheorem{theorem}{Theorem}
\newtheorem{lemma}[theorem]{Lemma}
\newtheorem{corollary}[theorem]{Corollary}

\theoremstyle{definition}

\theoremstyle{remark}

\newcommand{\set}[1]{\left\{#1\right\}}
\newcommand{\abs}[1]{\lvert#1\rvert}
\newcommand{\nm}[1]{\lVert#1\rVert}
\newcommand{\bnm}[1]{\big\lVert#1\big\rVert}

\newcommand{\D}{\mathbb{D}}

\newcommand{\N}{\mathbb{N}}

\newcommand{\R}{\mathbb{R}}

\newcommand{\C}{\mathbb{C}}

\newcommand{\ch}{\mbox{\large $\chi$}}

\renewcommand{\phi}{\varphi}

\renewcommand{\H}{\mathcal{H}}

\begin{document}

\title[Inner Functions in certain Hardy-Sobolev Spaces]{Inner Functions in certain Hardy-Sobolev Spaces }
\thanks{The first author is supported  in part by the Academy of Finland \#286877 and
the Finnish Academy of Science and Letters (Vilho, Yrj\"o and Kalle V\"ais\"al\"a Foundation),
and the second author is supported in part by the Generalitat de Catalunya, grant 2014SGR 75, 
and the Spanish Ministerio de Ciencia e Innovaci\'on, grant MTM2014-51824-P}

\author{Janne Gr\"ohn}
\address{Department of Physics and Mathematics, University of Eastern Finland\\ 
\indent P.O. Box 111, FI-80101 Joensuu, Finland}
\email{janne.grohn@uef.fi}

\author{Artur Nicolau}
\address{Departament de Matem\`atiques, Universitat Aut\`onoma de Barcelona\newline 
\indent 08193, Bellaterra, Barcelona, Spain}
\email{artur@mat.uab.cat }

\date{\today}

\subjclass[2010]{Primary 30J10; Secondary 30H10}
\keywords{Blaschke product, Hardy space, inner function}

\begin{abstract}
For $1/2<p<1$, a description of inner functions whose derivative is in the Hardy space $H^p$ is given 
in terms of either their mapping properties or the geometric distribution of their zeros.
\end{abstract}

\maketitle


\section{Introduction and results}

Let $\H(\D)$ be the collection of holomorphic functions in the unit disc $\D$ of the 
complex plane $\C$. For $0<p<\infty$, the Hardy space $H^p$ consists of those $f\in \H(\D)$ for which
\begin{equation*}
  \nm{f}_{H^p} = \sup_{0<r<1} \left( \frac{1}{2\pi} \int_0^{2\pi}
    \abs{f(r e^{i\theta})}^p \, d\theta \right)^{1/p}<\infty.
\end{equation*}
A bounded holomorphic function $I\in\H(\D)$ is called inner if $|I(e^{i\theta})| = 1$ for almost 
every (a.e.) $e^{i\theta}\in\partial\D$, where $I(e^{i\theta})=\lim_{r\to 1^-} I(re^{i\theta})$
denotes the radial limit at the point $e^{i\theta}$. It is well-known that every inner function $I$ can
be represented as a product $I = \xi B S$, where $\xi$ is a unimodular constant, $B$ is a~Blaschke product,
and $S$ is a singular inner function; see \cite[p.~72]{G:2007}. 

In the seventies Ahern and Clark \cite{AC:1974,AC:1976} considered inner functions in 
Hardy-Sobolev spaces, that is, those inner functions whose derivative belongs
to a~certain Hardy space, which opened a new page in complex function theory. 
The only inner functions whose derivative is in $H^1$ are Blaschke products with finitely many zeros, 
since these are the only inner functions which extend continuously to the closed unit disc. So, 
given $0<p<1$, the problem consists of describing those inner functions whose derivative is in 
$H^p$. As shown by Cohn \cite{C:1983}, this is related to the  smoothness of functions in the 
corresponding model space. It is well known that $p=1/2$ is a breaking point in this problem. 
For instance, an~inner function whose derivative is in $H^p$ for some $p \geq 1/2$ must be 
a Blaschke product (\cite[Theorem~3]{AC:1974}) while the singular inner function 
$I(z) = \exp((1+z)/(z-1))$ has derivative in $H^p$ for any $p<1/2$. 

Let $0<q<\infty$ and $0<s<1$. The weighted Besov space $B_s^q$ consists of those 
functions $f\in\H(\D)$ for which
\begin{equation*}
  \|f\|_{B_s^q}^q = \int_{\D} |f'(z)|^q (1-|z|)^{(1-s)q-1}\, dm(z)<\infty,
\end{equation*}
where $dm(z)$ denotes the Lebesgue area measure. Let $B\in B^q_s$ be inner.
If $qs\geq 1/2$ then $B$ is Blaschke product \cite[p.~736]{A:1983}, while if $qs \geq 1$ then $B$ is
a~finite Blaschke product \cite[Theorem~3.1]{DGV:2002}. The intermediate cases $1/2<qs<1$, $1\leq q <\infty$,
can be characterized in terms of Hardy spaces. Actually, if $1\leq q <\infty$ and $1/2<qs<1$, an inner function 
is in $B_s^q$ if and only if its derivative belongs to $H^{qs}$.  Moreover, there exists a constant $C=C(q,s)>0$ such that 
for any inner function $B$,
\begin{equation} \label{eq:sobolev}
  C^{-1} \|B'\|_{H^{qs}} \leq \|B\|_{B_s^q} \leq C \, \|B' \|_{H^{qs}}. 
\end{equation}
See \cite[Theorem~2]{V:1984}, and also \cite[Theorem~13]{D:1994} and \cite[Theorem~7(b)]{G:1987}. Note 
that this characterization does not extend to the case $qs=1/2$ by \cite[Lemma~2]{AC:1976}. 
 
Fix $1/2<p<1$. If $B$ is a Blaschke product with zeros $\{z_n\} \subset\D$ satisfying
\begin{equation} \label{eq:protas}
  \sum_n (1-|z_n|)^{1-p} < \infty,
\end{equation}
then Protas proved that  $B'\in H^p$  \cite[Theorem~2]{P:1973}. Conversely, if $B$ is a~Blaschke product
for which $B'\in H^p$, then
\begin{equation*} 
  \sum_n (1-|z_n|)^{(1-p)/p} < \infty
\end{equation*}
by \cite[Theorem 8(i)]{AC:1974}. These results as well as the improvements due to Verbitski\u{\i} (\cite{V:1984}, 
see also \cite{GGJ:2011,J:1990}) do not provide a complete characterization of the Blaschke products $B$ 
with $B' \in H^p$. 

The following results concern the mapping properties of Blaschke products in Hardy-Sobolev spaces. Let 
$B$ be a Blaschke product with zeros $\{z_n \}$. In \cite[p. 28]{C:1950}, for any $0<p<1$, Carleson 
proved that the condition \eqref{eq:protas} holds if and only~if 
\begin{equation*}\label{eq:carleson}
  \int_{\D }\frac{ \log |B(z)|^{-1 }}{(1-|z|)^{1+p}} \, dm(z) < \infty .
\end{equation*}
For $1/2 < p < 1$, Cohn showed \cite[Theorem 3]{C:1983} that a Blaschke product $B$ has derivative 
in $H^p$ if and only if 
\begin{equation*}
  \int_{\gamma} \frac{|dz|}{(1-|z|)^{p}} < \infty , 
\end{equation*}
where $\gamma$ is a Carleson contour of $B$. See also \cite[Theorem 13]{D:1994}. Many of the results 
mentioned above can also be found in the monograph \cite{M:2013}. 

The main purpose of this paper is to present a description of Blaschke products whose derivative is 
in $H^{p}$ for $1/2<p<1$,  in terms of either their mapping properties or the geometric distribution of their zeros. 
For $1<\alpha<\infty$, 
\begin{equation*} 
  \Gamma_\alpha(e^{i\theta}) = \set{ z\in\D :  \abs{z-e^{i\theta}} < \alpha(1-\abs{z})}
\end{equation*}
denotes a Stolz angle with vertex at $e^{i\theta}\in\partial\D$. Given a Blaschke product $B$ 
with zeros $\{z_n\}$, repeated according to their multiplicity, consider the function 
\begin{equation} \label{eq:defF}
  F_{\alpha, B} (e^{i\theta}) = \sum_{z_n\in \Gamma_\alpha(e^{i\theta})} \, \frac{1}{1-|z_n|}, \quad e^{i\theta}\in\partial\D.
\end{equation}


\begin{theorem} \label{thm:main}
Let $1/2<p<1$, $1<\alpha<\infty$ and $B$ be a Blaschke product
with zeros $\{z_n\}$. Then the following conditions are equivalent:

\begin{enumerate}
\item[\rm (a)] 
$B'\in H^p$; 

\item[\rm (b)]
There exists a constant $0<c<1$
such that
\begin{equation} \label{eq:intcond}
  I(c)= \int_{\{ z\in\D \,:\, |B(z)|<c\}} \frac{dm(z)}{(1-|z|)^{1+p}} < \infty;
\end{equation}

\item[\rm (c)]
$F_{\alpha, B} \in L^p(\partial\D)$.
\end{enumerate}
Moreover, there exist positive constants $C_1 = C_1(c,p)$ and $C_2=C_2(\alpha , p)$
such that if any of the conditions above holds then $C_1^{-1} I(c)^{1/p} \leq \nm{B'}_{H^p} \leq C_1\, I(c)^{1/p}$ and  
$C_2^{-1} \nm{F_{\alpha, B}}_{L^p(\partial\D)} \leq \nm{B'}_{H^p} \leq C_2 \, \nm{F_{\alpha, B}}_{L^p(\partial\D)}$.
\end{theorem}

The proof of Theorem~\ref{thm:main} shows that, if (b) holds for a certain $0<c<1$,
it is also satisfied for any $0<c<1$. Similarly, 
if (c) is satisfied for a certain $\alpha >1$, it also holds for any other choice of $\alpha>1$. Observe that 
Theorem \ref{thm:main} does not hold for $p \leq 1/2$. Actually, if $B$ is a Blaschke product with a~single zero 
$z_0 \in \D$ and $F=F_{\alpha, B}$ is the corresponding function defined in \eqref{eq:defF}, both 
$\nm{F}_{L^p(\partial\D)}^p$ and $I(c)$ are comparable to $ (1-|z_0|)^{1-p}$ for any $0<p<1$, while $\nm{B'}_{H^p}^p$ 
is comparable to $(1-|z_0|)^p$ if $p< 1/2$ and to $(1-|z_0|)^{1/2} (1+ \ln(1-|z_0|)^{-1} )$ if $p=1/2$. 

Recall that $\{z_n\} \subset \D$ is separated (with respect to the pseudo-hyperbolic metric) if
there exists a~constant $\delta=\delta(\{z_n\}) >0$ such that $\varrho(z_n,z_k) >\delta$ for all $ n\neq k$.
Here $\varrho(z,w)=|z-w|/|1-\overline{z}{w}|$ denotes the pseudo-hyperbolic distance between
the points $z,w\in\D$. If the zeros, counting multiplicities, of a Blaschke product form a separated sequence, the 
function $F$ can be estimated pointwise and the condition (c) in Theorem \ref{thm:main} is equivalent to \eqref{eq:protas}.


\begin{corollary} \label{cor:cohnimp}
If $1/2<p<1$ and $B$ is a Blaschke product with
separated zeros~$\{z_n\}$, then $B'\in H^p$ if and only if 
$\{z_n\}$ satisfies \eqref{eq:protas}.
\end{corollary}

Although Corollary~\ref{cor:cohnimp} can be deduced from our main result,
we will give an~independent proof for it. 
Corollary~\ref{cor:cohnimp} admits an immediate generalization to finite unions of separated 
sequences and was originally proved by Protas in \cite[Theorem~3]{P:2011}. 
The following result, due to Verbitski\u{\i},
see \cite[Theorem~K]{GGJ:2011} and \cite[Theorem~5(b)]{V:1984}, arises as another consequence 
of Theorem~\ref{thm:main}.


\begin{corollary} \label{cor:m1}
Let $1/2<p<1$ and  $1<\beta<\infty$. Let $B$ be a Blaschke product whose zeros 
$\{z_n\}$ are contained in a Stolz angle $\Gamma_\beta(e^{i\theta})$ for some $e^{i\theta}\in\partial\D$. 
Consider $N_j = \#\{z_n : 2^{-j} < 1- |z_n| \leq 2^{-j+1} \}$ for $j\in\N$.
Then $B'\in H^p$ if and only if 
\begin{equation*}
  \sum_{j=1}^\infty 2^{-j(1-p)} N_j^p < \infty.
\end{equation*}
Moreover, there exists a constant $C=C(\beta , p) >0$ such that 
\begin{equation*}
  C^{-1} \sum_{j=1}^\infty 2^{-j(1-p)} N_j^p 
  \, \leq \nm{B'}_{H^p}^p \leq 
  C \, \sum_{j=1}^\infty 2^{-j(1-p)} N_j^p.
\end{equation*}
\end{corollary}

As a consequence of Theorem \ref{thm:main} we obtain a proof for the following 
beautiful result proved in \cite[Theorem 6.2]{A:1979} and \cite[Theorem 4.3]{C:1950}.


\begin{corollary} \label{cor:m2}
Let $1/2<p<1$ and $B$ be a Blaschke product. Then $B' \in H^p$ if and only if there exists a point 
$a \in \D$ such that $(a-B)/(1-\overline{a}B)$ is a Blaschke product and its zeros $B^{-1} (\{a\}) = \{z_n (a) \}$ satisfy 
\begin{equation}\label{eq:preimages}
  \sum_n (1-|z_n (a)|)^{1-p} < \infty.
\end{equation}
Moreover, if $B' \in H^p$ then for all $a \in \D$, except possibly for a set of logarithmic capacity zero, 
the preimages $B^{-1} (\{a\}) = \{z_n (a) \}$ satisfy \eqref{eq:preimages}. 
\end{corollary}

Consider the collection of dyadic arcs $\{ e^{i\theta} \in \partial\D :  (j-1)2^{-n+2} \pi \leq \theta \leq j 2^{-n+2} \pi \}$, 
$ j=1,\dotsc, 2^{n-1}$, $n \in \N$, of the unit circle. Given an arc $I \subset \partial \D$ consider the sector 
$Q= Q(I) = \{ re^{i\theta}  : e^{i\theta}\in I, \, 1-|I|\leq r < 1 \}$, whose base is the arc $I$,  and its top part 
$T( Q(I) ) = \big\{ re^{i\theta}  : e^{i\theta}\in I, \, 1-|I|\leq r <  1- |I|/2 \big\}$.  Here $|I| = \ell(Q)$
denotes the normalized arc length of $I$. Let $\mathcal{E}$ be the collection of dyadic sectors of $\D$, that is, 
sectors whose base is a dyadic arc of $\partial\D$. 

Part (c) of Theorem \ref{thm:main} gives a description of Blaschke products in certain Hardy-Sobolev spaces in terms 
of the function $F_{\alpha, B}$ which is defined on the unit circle. There is a similar description in $\D$. 
Given a sequence $\{z_n \}\subset \D$ and a sector~$Q\in\mathcal{E}$, let 
$N(Q)$ be the number of points in $\{z_n \} \cap T(Q)$, that is, $N(Q) = \# \{n: z_n \in  T(Q) \} $. 
Consider the function
\begin{equation*}
  F(Q)= \sum \frac{N(R)}{\ell (R)}, \quad Q\in\mathcal{E},
\end{equation*}
where the sum is taken over all sectors $R\in\mathcal{E}$ for which $\ell(R)\geq \ell(Q)$ and $Q\subset 3R$.
Here $3R$ is a sector whose base is concentric to that of $R$ and  $\ell(3R)=3 \, \ell(R)$.
For $N \in \N$, let ${\mathcal{E}}_N$ 
be the collection of maximal sectors $Q \in \mathcal{E}$ such that $F(Q)> 2^N$. Observe that the 
maximality gives $N(Q)>0$ for any $Q \in {\mathcal{E}}_N$. We denote 
$\ell ({\mathcal{E}}_N) = \sum \ell (Q)$, where the sum runs over all sectors 
$Q\in{\mathcal{E}}_N$. Our next result is a reformulation of part (c) of Theorem \ref{thm:main}.


\begin{corollary} \label{thm:F}
Let $1/2<p < 1$ and $B$ be a Blaschke product with zeros~$\{z_n\}$. Then $B'\in H^p$ if and only if 
\begin{equation} \label{eq:F}
  \sum_{N=1}^\infty 2^{Np} \,  \ell({\mathcal{E}}_N) <   \infty.
\end{equation}
\end{corollary}

Given a sector $Q_0\in \mathcal{E}$ and $0< \varepsilon < 1$, 
let $\mathcal{A}_{\varepsilon} (Q_0) $ be the family of those dyadic sectors $Q \subset Q_0$, $Q \neq Q_0$, for which
\begin{equation*}
  (1- \varepsilon) \, \frac{N(Q_0)}{\ell(Q_0)} \leq  \frac{N(Q)}{\ell(Q)} \leq (1+ \varepsilon) \, \frac{N(Q_0)}{\ell(Q_0)} . 
\end{equation*}
Part (a) of the next result provides a sufficient condition for a Blaschke product to be in a Hardy-Sobolev space,  
which as stated in part (b), is also necessary under a suitable assumption.


\begin{corollary} \label{thm:protasimp}
Let $1/2<p < 1$ and $B$ be a Blaschke product with zeros $\{z_n\}$.

\begin{enumerate}
\item[\rm (a)]
 If 
\begin{equation} \label{eq:protasrep}
  \sum_{Q \in \mathcal{E}} N(Q)^{p} \ell(Q)^{1-p} < \infty,
\end{equation}
then $B'\in H^p$.

\item[\rm (b)]
Assume there exist constants $C>0 $ and $0< \varepsilon < 1$ 
such that for any sector $Q_0\in\mathcal{E}$ with $N(Q_0) \geq C$,
\begin{equation*} 
  \sum_{Q \in \mathcal{A}_{\varepsilon} (Q_0)} \ell(Q) \leq  C \, \ell(Q_0).
\end{equation*}
If $B' \in H^p$, then the estimate \eqref{eq:protasrep} holds. 
\end{enumerate}
\end{corollary}

Part (a) of Corollary~\ref{thm:protasimp} generalizes  \cite[Theorem~2]{P:1973} and \cite[Theorem~4(a)]{V:1984}, 
while part (b) improves Corollary \ref{cor:cohnimp}. 
The following corollary concerns those Blaschke products for which the collection $\mathcal{A}_{1/2}(Q_0)$
is empty for any $Q_0$ with $N(Q_0)\geq 1$.
This result emerges as an immediate consequence of part~(b) of Corollary~\ref{thm:protasimp}.


\begin{corollary}
Let $1/2<p<1$ and $B$ be a Blaschke product with zeros $\{z_n\}$. 
Assume that for any $Q_0\in\mathcal{E}$ with $N(Q_0)\geq 1$, and any dyadic subsector $Q\subset Q_0$,
either $N(Q)=0$ or $N(Q)\geq N(Q_0)$. Then, $B'\in H^p$ if and only if \eqref{eq:protasrep} holds.
\end{corollary}

We briefly sketch the main ideas in the proof of Theorem \ref{thm:main}. The necessity of~\eqref{eq:intcond} 
follows easily from the estimate \eqref{eq:sobolev} but the sufficiency is harder. Given a Blaschke product 
$B$ satisfying \eqref{eq:intcond},
we consider the family $\mathcal{F}$ of dyadic sectors $Q$ such that 
$T(Q) \cap \{z\in \D : |B(z)| \leq c \} \neq \emptyset$ and the Blaschke product $B^{\star}$ 
with zeros  $\{z_Q : Q \in \mathcal{F} \}$.
Here $z_Q$ denotes the center of $T(Q)$. It turns out that there exists 
a constant $M=M(c)$ such that $|B'(e^{i \theta})| \leq M \, |{B^\star}' (e^{i \theta})|$ for 
a.e.~$e^{i \theta} \in \partial \D$. The condition \eqref{eq:intcond} easily implies  ${B^\star}' \in H^p$ 
and hence $B' \in H^p$. Let us now discuss the statement in part (c). 
Let $P_{z_n}$ denote the Poisson kernel at the point $z_n \in \D$,
and let $T_{z_n}$ be the box kernel given by 
$T_{z_n} (e^{i \theta})=(1-|z_n|)^{-1}$, if $z_n \in \Gamma_\alpha (e^{i \theta})$,
and $T_{z_n}(e^{i\theta})=0$ otherwise.
It was proved by Frostman \cite{F:1942} (see also \cite{AC:1974})
that, if $B$ is a Blaschke product with zeros $\{z_n \}$, then  
$B' \in H^p$ if and only if $\sum_n P_{z_n} \in L^p (\partial \D)$. The necessity of part (c) in Theorem~\ref{thm:main} follows 
from the trivial estimate $T_{z_n} (e^{i \theta}) \leq C(\alpha) \, P_{z_n} (e^{i \theta})$, which holds  
for all $e^{i \theta} \in \partial \D$. The sufficiency 
is harder and can be deduced from a result of Luecking \cite[Proposition 1]{L:1991}. See also~\cite[Lemma~C]{P:2016}. 
However the proof  of Luecking's result uses the atomic decomposition of tent spaces 
and we present a~more elementary argument. We first 
split the sequence of zeros into convenient families defined according to the size of the 
distribution function of $F_{\alpha, B}$. 
In each family we consider the corresponding Blaschke product $b$ and prove an estimate 
for the size of the set of points 
$z \in \D$ such that  $|b(z)| < 1/2$ by using a~dyadic construction and certain stopping time arguments. It turns out that, 
as Luecking's result, our  argument also extends to the higher dimensional setting. 

The paper is organized as follows. In Section 2 we collect some elementary properties of inner functions. 
Section 3 contains the proof of Theorem~\ref{thm:main}, while Section~\ref{sec:prfcor} 
is devoted to the proofs of the corollaries. In the concluding Section~\ref{sec:questions},
we pose two closely related questions.


\section{Auxiliary results}

We begin by recalling basic properties of Blaschke products.
If $\{z_n\} \subset \D$, then 
\begin{equation} \label{eq:lemmaa}
\begin{split}
  \int_0^{2\pi} \# \!\left( \{z_n\} \cap \Gamma_\alpha(e^{i\theta}) \right)  d\theta
  \, \simeq  \, \sum_n (1-|z_n|)
\end{split}
\end{equation}
with comparison constants depending on $\alpha$.
The notation  $C_1 \lesssim C_2$ means that $C_1/C_2$ is bounded by a positive constant, while 
$C_1 \simeq C_2$ indicates that $C_1/C_2$ is bounded from above and below by two positive constants.
If $\{z_n\}$ is a~Blaschke sequence, that is $\sum_n (1-|z_n|)<\infty$, then
$\{z_n\} \cap \Gamma_\alpha(e^{i\theta})$ contains only finitely many points for a.e.~$e^{i\theta}\in\partial\D$.

If $\{z_n\}$ is a Blaschke sequence, then the Blaschke product
\begin{equation*}
  B(z) =  \prod_n \frac{|z_n|}{z_n} \, \frac{z_n-z}{1-\overline{z}_n z}, 
  \quad z\in\D,
\end{equation*}
is a bounded analytic function in $\D$ having zeros precisely at the points $z_n$. Here 
we take the convention that each zero is repeated according to its multiplicity and that  $|z_n|/z_n=1$ for $z_n=0$. Now
\begin{equation*} 
  \frac{B'(z)}{B(z)} 
  = - \sum_n \frac{1-|z_n|^2}{(z_n-z)(1-\overline{z}_n z)},
  \quad z\in\D \setminus \{ z_n\}.
\end{equation*}
Since $|B(z)| \leq |z_n - z|/|1-\overline{z}_n z|$ for $z\in\D$, and for any $n$, we deduce 
\begin{equation} \label{eq:bderind}
  |B'(z)| 
  \leq  \sum_n \frac{1-|z_n|^2}{|1-\overline{z}_n z|^2},
  \quad z\in\D.
\end{equation}

A holomorphic mapping $f$ from the unit disc into itself is said to have a finite angular derivative at 
$e^{i\theta}\in\partial\D$ if $f(e^{i\theta}) = \lim_{r\to 1^-} f(re^{i\theta})$
exists, $|f(e^{i\theta})|=1$, and $f'(e^{i\theta}) = \lim_{r\to 1^-} f'(re^{i\theta})$ exists.
By the Julia-Carath\'eodory theorem \cite[p.~57]{S:1993}, $f$ has a finite  angular derivative at the point $e^{i \theta} \in \partial \D$ 
if and only if
\begin{equation*}
  \liminf_{z\to e^{i\theta}} \frac{1-|f(z)|}{1-|z|} < \infty, 
\end{equation*}
and in this case, 
\begin{equation*}
  |f'(e^{i \theta})| = \liminf_{z\to e^{i\theta}} \frac{1-|f(z)|}{1-|z|}. 
\end{equation*}

A classical result of Frostman says that the Blaschke product $B$ with zeros $\{z_n\}$ 
has a finite angular derivative at 
$e^{i\theta}\in\partial\D$ if and only~if
\begin{equation*} 
  \sum_n \frac{1-|z_n|^2}{|z_n - e^{i\theta}|^2} < \infty.
\end{equation*}
Moreover, if $B$ has a finite angular derivative at $e^{i\theta}\in\partial\D$, then
\begin{equation} \label{eq:Bprimerep}
  \big| B'(e^{i\theta}) \big| 
  = \sum_n P_{z_n}(e^{i\theta})
  = \sum_n \frac{1-|z_n|^2}{|z_n - e^{i\theta}|^2}.
\end{equation}
See \cite[Th\'eor\`eme VI]{F:1942} and \cite[Theorem~2]{AC:1974}. 
Here $P_{z_n}$ is the Poisson kernel at the point $z_n\in\D$. 
As usual, we write $|B'(e^{i\theta})|=\infty$ whenever $B$ has no angular derivative at $e^{i\theta}\in\partial\D$. With
this convention \eqref{eq:Bprimerep} holds for all $e^{i\theta}\in\partial\D$. 

We use the following auxiliary result which is closely related to an idea of Cohn \cite[Lemma~1]{C:1983}.


\begin{lemma} \label{lemma:ccc}
Let $B$ and $B^\star$ be inner functions. Assume that
there exist constants $0<\sigma<1$ and $0<C<\infty$ such that
\begin{equation*}
  \log\frac{1}{|B(z)|} \leq C \, \log\frac{1}{|{\displaystyle B^\star}(z)|}
\end{equation*}
whenever $z\in\D$ and $|B^\star(z)|\geq \sigma$.
If ${\displaystyle B^\star}'\in H^p$ for some $0<p<\infty$, then $B'\in H^p$.
Moreover, $\nm{B'}_{L^p(\partial\D)} \leq C \sigma^{-1} \nm{{B^\star}'}_{L^p(\partial\D)}$.
\end{lemma}


\begin{proof}
As long as $|B^\star(z)|\geq \sigma$, we have
\begin{equation*} 
  1-|B(z)| \leq \log\frac{1}{|B(z)|} \leq C \, \log\frac{1}{|{\displaystyle B^\star}(z)|} \leq \frac{C}{\sigma}\big(  1-|B^\star(z)| \big).
\end{equation*}
Since ${B^\star}' \in H^p$ for some
$0<p<\infty$, the radial limit $|{B^\star}'(e^{i\theta})|$ exists and is finite for a.e.~$e^{i\theta}\in\partial\D$.
Let $e^{i\theta}\in\partial \D$ be any point where $B^\star$ has a finite angular derivative. By the Julia-Carath\'eodory theorem,
\begin{equation*}
  \liminf_{z\to e^{i\theta}} \, \frac{1-|B^\star(z)|}{1-|z|} = |{B^\star}'(e^{i\theta})|.
\end{equation*}
Hence 
\begin{equation*}
  \liminf_{z\to e^{i\theta}} \frac{1-|B(z)|}{1-|z|} \leq \frac{C}{\sigma} \, |{B^\star}'(e^{i\theta})|.
\end{equation*}
Again, by the Julia-Carath\'eodory theorem, $B'(e^{i\theta})$ exists, it is finite and satisfies 
$|B'(e^{i\theta})| \leq C \sigma^{-1} |{B^\star}'(e^{i\theta})|$.
Since this argument holds for a.e.~$e^{i\theta}\in\D$, we deduce
$\nm{B'}_{L^p(\partial\D)} \leq C \sigma^{-1}  \nm{{B^\star}'}_{L^p(\partial\D)}$.
\end{proof}


\section{Proof of Theorem~\ref{thm:main}}

We first prove the equivalence between the statements (a) and (b). 
Assume that $B$ is a Blaschke product for which $B'\in H^p$ for $1/2<p<1$. If $0<c<1$, then
\begin{align*}
  \int_{\{ z\in\D \, : \, |B(z)|<c\}} \frac{dm(z)}{(1-|z|)^{1+p}} 
  & \leq \frac{1}{1-c} \, \int_{\D} \frac{1-|B(z)|}{(1-|z|)^{1+p}} \,  dm(z). 
\end{align*}
Since 
\begin{equation*}
  1-|B(r e^{i \theta})| \leq \int_r^1 |B'(s e^{i \theta})| \, ds, \quad 0<r<1, \quad \text{a.e. } e^{i \theta} \in \partial \D,
\end{equation*}
Fubini's theorem gives
\begin{equation*}
  \int_{\{ z\in\D \, : \, |B(z)|<c\}} \frac{dm(z)}{(1-|z|)^{1+p}} 
  \leq \frac{p^{-1}}{1-c} \int_{\D} |B'(z)| (1-|z|)^{-p} \, dm(z) .
\end{equation*}
Estimate \eqref{eq:sobolev} with $q=1$ and $s=p$ gives that the last integral is finite, and hence $I(c)< \infty$. 

Conversely, given $0<c<1$, we will prove that there exists a constant $M=M(c)>0$ such that, 
if $B$ is a Blaschke product with $I(c)< \infty$ then
\begin{equation} \label{eq:ine}
  \|B'\|_{H^p}^p \leq M \int_{\{ z\in\D \, : \, |B(z)|<c\}} \frac{dm(z)}{(1-|z|)^{1+p}} .
\end{equation}
It suffices to prove \eqref{eq:ine} for finite Blaschke products $B$, provided that 
$M$ is independent of $B$. For $n \geq 0$, consider the collection $\mathfrak{D}_n$ of  the $2^n$ dyadic 
subarcs of $\partial\D$, of equal normalized length $2^{-n}$, having pairwise disjoint interiors, and obtain
\begin{equation} \label{eq:dun}
  \D = \bigcup_{n=0}^\infty \bigcup_{I \in \mathfrak{D}_n} T\big(Q(I) \big).
\end{equation}
Partition each top part $T(Q(I))$ in \eqref{eq:dun} into $4^k$ {\it squares} of
with pairwise disjoint interiors
in such a way that the pseudo-hyperbolic diameter of each {\it square} is at most $c/2$.
This is accomplished by choosing $k=k(c)\in\N\cup \{0\}$ to be sufficiently large.
After reindexing the sets, we deduce $\D = \bigcup_{n=1}^{\infty} T_n$
where each set $T_n$ has pseudo-hyperbolic diameter smaller than $c/2$.

Consider the sequence ${\{z_n^\star\}} \subset \D$
formed by the centers of those sets $T_n$, $n\in\N$,
for which $T_n \subset \{z\in\D: |B(z)|<c\}$. By choosing sufficiently small $\mu=\mu(c)$ with $0<\mu<1$,
the Euclidean discs $D_n$ of radius $\mu (1-|z_n^\star|)$,
centered at $z_n^\star$, are pairwise disjoint and satisfy $D_n \subset T_n$ for all $n$.  Now
\begin{equation*}
  \sum_n (1-|z_n^\star|)^{1-p}
  \lesssim \, \sum_n \, \int_{D_n}  \frac{ dm(z)}{(1-|z|)^{1+p}}
  \leq \int_{\{z\in\D \, : \, |B(z)|<c\}}  \frac{ dm(z)}{(1-|z|)^{1+p}}
\end{equation*}
with a comparison constant depending on $c$ and $p$. 
Let $B^\star$ be the Blaschke product with zeros ${\{z_n^\star\}}$.
By \eqref{eq:protas}, we deduce ${B^\star}' \in H^p$ and 
\begin{equation*}
  \|{B^\star}'\|_{H^p}^p \lesssim \int_{\{ z\in\D \, : \, |B(z)|<c\}} \frac{dm(z)}{(1-|z|)^{1+p}} .
\end{equation*}

We proceed to show that $B'\in H^p$. Consider the set
$\Omega = \bigcup T_n$, where the union runs over the indices $n\in\N$ for which 
$T_n \subset \{z\in\D :  |B(z)|<c\}$. If $z\not\in\Omega$ then 
$z$ belongs to some set $T_n$ which contains 
another point $w_z$ such that $|B(w_z)|\geq c$. By the Schwarz-Pick lemma,
\begin{equation*} 
  |B(z)| \geq \frac{|B(w_z)|- \varrho(z,w_z)}{1-\varrho(z,w_z) |B(w_z)|} 
  \geq c_1= \frac{c-c/2}{1-c^2/2},
  \quad z\in\D \setminus \Omega.
\end{equation*}
Note that $|B^\star(z)| < \sigma$ for all $z\in \overline{\Omega} \cap \D$, where $\sigma=\sigma(c)$
with $0<\sigma<1$ is a~sufficiently large constant, 
since $B^\star$ has a zero within a fixed pseudo-hyperbolic distance
from any point on $\overline{\Omega}\cap\D$. Hence, if $|B^\star (z)| \geq \sigma$ we have $|B(z)| \geq c_1$, 
and Lemma~\ref{lemma:ccc} yields $\|B'\|_{H^p}^p \lesssim \|{B^\star}'\|_{H^p}^p$.

Let us now prove the equivalence between the statements (a) and (c). Since $\alpha$ is fixed,
we write $F=F_{\alpha, B}$. Assume that $B$ is a Blaschke product with zeros $\{z_n\}$, and
$B'\in H^p$ for $1/2<p<1$. Since
\begin{equation*}
  P_{z_n}(e^{i\theta}) = \frac{1-|z_n|^2}{|z_n-e^{i\theta}|^2} > \frac{1}{\alpha^2} \, \frac{1}{1-|z_n|},
  \quad z_n\in \Gamma_\alpha(e^{i\theta}), \quad e^{i\theta}\in\partial\D,
\end{equation*}
we obtain
\begin{equation*}
  \sum_{z_n \in \Gamma_\alpha(e^{i\theta})} P_{z_n}(e^{i\theta}) 
  \geq \frac{1}{\alpha^2}\sum_{z_n \in \Gamma_\alpha(e^{i\theta})}  \frac{1}{1-|z_n|}
  = \frac{1}{\alpha^2} \, F(e^{i\theta}),
  \quad e^{i\theta}\in \partial\D.
\end{equation*}
By assumption $B$ has a finite angular derivative for a.e.~$e^{i\theta}\in\partial\D$, 
and hence by \eqref{eq:Bprimerep}, $\nm{F}_{L^p(\partial\D)}^p \leq \alpha^{2p} \, \nm{B'}_{H^p}^p$.

The converse assertion is more difficult. We begin by reducing the problem to a~certain estimate,
which is then proved in steps. Assume that $B$ is a~Blaschke product and $F\in L^p(\partial\D)$. 
For any arc $I\subset\partial\D$,
we consider a dilated sector 
\begin{equation*}
  Q_\alpha(I) = \big\{ re^{i\theta} : e^{i\theta} \in I, \, \max\{ 1-c(\alpha) |I|, 0 \}\leq r < 1 \big\},
\end{equation*}
where $c(\alpha)$ is a positive constant which will be defined later.
In this case we write $\ell(Q_\alpha) = 1 - \max\{ 1-c(\alpha) |I|, 0 \}$, where $|I|$ 
denotes the normalized arc length of $I$ as above. Given an arc $I\subset \partial\D$ and $n \geq 0$, 
consider the dyadic decomposition of $I$ into the collection  $ \mathfrak{D}_n $ of $2^n$ subarcs of $I$, 
of equal length, having pairwise disjoint interiors, and obtain
\begin{equation*}
  Q_\alpha(I) = \bigcup_{n=0}^\infty \bigcup_{J \in \mathfrak{D}_n} T\big(Q_\alpha(J) \big).
\end{equation*}

Consider the open sets 
\begin{equation*}
  \big\{ e^{i\theta} \in \partial\D : F(e^{i\theta}) > 2^N \big\} \, = \, \bigcup_j \, I_j^{(N)}, \quad N\in\N,
\end{equation*}
where $\{I_j^{(N)} \}$ are pairwise disjoint open subarcs
of $\partial\D$ for any fixed $N \in \N$. These collections are nested in the sense that
\begin{equation*}
  \bigcup_j \, I_j^{(N+1)} \subset \, \bigcup_j \, I_j^{(N)}, \quad N\in\N.
\end{equation*}
Write $Q_j^{(N)} = Q_\alpha\big( I_j^{(N)} \big)$ for short. Decompose the zero-sequence $\{z_n\}$ of
$B$ in the following way:  
\begin{equation*}
  \{z_n\} 
  = \bigcup_{N=1}^\infty \bigcup_{j} \Lambda_{N,j}, \quad \Lambda_{N,j} 
  = \{z_n\} \cap \left( \, Q_j^{(N)} \, \setminus \, \bigcup_k \, Q_k^{(N+1)} \right).
\end{equation*}
Following this notation 
\begin{equation*}
  \sum_n P_{z_n}(e^{i\theta}) = \sum_{N,j} \sum_{z_n \in \Lambda_{N,j}} \!P_{z_n}(e^{i\theta}),
  \quad e^{i\theta}\in\partial\D,
\end{equation*}
and since $1/2<p<1$,
\begin{equation} \label{eq:triplesumP}
  \bigg\lVert \sum_n P_{z_n} \bigg\rVert_{L^p(\partial\D)}^p
  \leq  \sum_{N,j} \, \bigg\lVert  \sum_{z_n \in \Lambda_{N,j}} 
  \!P_{z_n} \, \bigg\rVert_{L^p(\partial\D)}^p.
\end{equation}
The proof reduces to showing that there exists a constant $C=C(\alpha,p)>0$ such that
\begin{equation} \label{eq:(2)}
  \bigg\lVert  \sum_{z_n \in \Lambda_{N,j}} 
  \!P_{z_n} \, \bigg\rVert_{L^p(\partial\D)}^p \leq C \, 2^{Np} \, |I_j^{(N)} |
\end{equation}
for all $j$ and $N$.

Assume that \eqref{eq:(2)} holds, and let us complete the proof. By \eqref{eq:triplesumP} and \eqref{eq:(2)}, 
\begin{align*}
  \bigg\lVert \sum_n P_{z_n} \bigg\rVert_{L^p(\partial\D)}^p
  &  \leq C \, \sum_{N=1}^\infty 2^{Np} \bigg( \sum_j |I_j^{(N)}| \bigg)\\
  &  = C \, \sum_{N=1}^\infty 2^{Np} \, \big|\big\{ e^{i\theta} \in \partial\D : F(e^{i\theta}) > 2^N \big\}\big|\\
  &  \leq \frac{2C}{p} \int_0^{2\pi}  | F(e^{i\theta}) |^p \, d\theta . 
\end{align*}
We conclude that $B'\in H^p$ by \eqref{eq:Bprimerep} and \cite[Theorem~4]{AC:1974}, and moreover,
\begin{equation*}
  \nm{B'}_{H^p}^p \leq \frac{4 \pi C}{p} \, \nm{F}_{L^p(\partial\D)}^p.
\end{equation*}


\subsection*{Proof of \eqref{eq:(2)}} 
It is clear that the points in $\Lambda_{N,j}$, 
for fixed $j$ and $N$, are uniformly bounded away from $\partial\D$.
The remaining part of the proof of Theorem~\ref{thm:main} consists of three steps. 
In the first step we obtain a uniform estimate for the  function $F$ corresponding to 
the points in $\Lambda_{N,j}$, while in the second step we 
consider the Blaschke product $b$ with zeros $\Lambda_{N,j}$.
Finally, in the third step, we construct another finite Blaschke product $b^\star$
 in such a way that $b^\star$ has a~zero within a~fixed hyperbolic distance from any point where $b$ is
 (in a certain sense) small, and show that the radial limits of $b'$ are majorized by those of ${b^\star}'$.


\subsubsection*{Step 1.}
Fix $j$ and $N$. To simplify the notation, we write
\begin{equation*}
  I_0 = I_j^{(N)}, \quad Q_0 = Q_j^{(N)}, \quad 
  \Lambda = \Lambda_{N,j} .
\end{equation*}
Note that $\Lambda$ contains only finitely many points. Consider the function
\begin{equation*} 
  G(e^{i\theta}) = \sum_{z_n \in \Lambda \,\cap \,\Gamma_\beta(e^{i\theta})} \frac{1}{1-|z_n|}, \quad e^{i\theta}\in \partial\D,
\end{equation*}
where the constant $\beta=\beta(\alpha)$ with $1<\beta<\alpha$ is chosen in such a way that the aperture 
of $\Gamma_\alpha$ is twice the aperture of $\Gamma_\beta$. We proceed to prove
\begin{equation} \label{eq:(3)}
  G(e^{i\theta}) \leq 2^{N+1}, \quad e^{i\theta} \in \partial\D.
\end{equation}
If, on contrary to \eqref{eq:(3)}, there exists $e^{i\theta_0}\in\partial\D$ such that $G(e^{i\theta_0})>2^{N+1}$, then there
is a point $z_n^\star \in \Lambda \,\cap \,\Gamma_\beta(e^{i\theta_0})$ such that
\begin{equation*}
  \sum_{\substack{z_n \in \Lambda \,\cap \,\Gamma_\beta(e^{i\theta_0}) \\ 1-|z_n| \, \geq \, 1-|z_n^{\star}|}} \frac{1}{1-|z_n|} > 2^{N+1}.
\end{equation*}
It follows that
\begin{equation*}
  F(e^{i\theta}) > 2^{N+1}, \quad \left| \, e^{i\theta} - z_n^\star / |z_n^\star| \, \right| < t(\alpha) (1-|z_n^\star|),
\end{equation*}
where $t(\alpha) > 0$.
Hence, the arc of Euclidean length $2\, t(\alpha) \, (1-|z_n^\star|)$, centered at $z_n^\star / |z_n^\star|\in\partial\D$, 
is contained in $\bigcup_k I_k^{(N+1)}$. Define $c(\alpha) = \pi\, t(\alpha)^{-1}$.
It follows that $z_n^\star \in \bigcup_k Q_k^{(N+1)}$, which is clearly a~contradiction. This proves \eqref{eq:(3)}.


\subsubsection*{Step 2.}

Let $b=b(N,j)$ be the (finite) Blaschke product with zeros $\Lambda = \Lambda_{N,j}$. 
Since top parts of (dilated) dyadic sectors have bounded hyperbolic diameter, by the Schwarz-Pick lemma, 
there exists a constant $\kappa=\kappa(\alpha)$ with $0<\kappa<1$
satisfying the following property: if $|b(z)|<1/2$ for  some point $z\in Q_0$,
then $|b(z_Q)|<\kappa$ where $Q$ is the dyadic subsector of $Q_0$ with $z\in T(Q)$.
Let $\mathcal{A}$ be the family of 
those proper (dilated) dyadic subsectors $Q$ of $Q_0$ for which $|b(z_Q)|<\kappa$. This step 
consists of proving the estimate
\begin{equation} \label{eq:(4)}
  \sum_{Q\in \mathcal{A}} \ell(Q)^{1-p} \leq c_1 \, 2^{Np} \, |I_0|,
\end{equation}
where $c_1= c_1(\alpha,p)$ is a positive constant. 

The estimate \eqref{eq:(3)} yields $1-|z_n|\geq 2^{-N-1}$ for any $z_n\in \Lambda$. Consequently,
if $z\in\D$ and $1-|z|\leq 2^{-N-2}$, then
\begin{equation*} 
  \inf_{z_n \in \Lambda} \, \varrho(z,z_n)
  \geq \inf_{z_n \in \Lambda} \, \frac{|z|-|z_n|}{1-|z_n| \, |z|} 
  > \frac{1}{3}.
\end{equation*}
 Using that $\log x^{-1} \simeq 1-x^2$ if $1/3 < x <1$, we deduce
\begin{equation} \label{eq:long}
  \log\frac{1}{|b(z)|} 
  \simeq (1-|z|) \, \sum_{z_n\in\Lambda} \frac{1-|\overline{z}_n z|^2}{|1-\overline{z}_n z|^2},
  \quad 1-|z|\leq 2^{-N-2},
\end{equation}
with absolute comparison constants. Define
\begin{equation*} 
  v(z) = \sum_{z_n\in\Lambda} \frac{1-|\overline{z}_n z|^2}{|1-\overline{z}_n z|^2}, \quad z\in\D.
\end{equation*}
It is clear that $v$ is a positive harmonic function in $\D$. By \eqref{eq:long},
\begin{equation} \label{eq:(5)}
  \log\frac{1}{|b(z)|}  \simeq (1-|z|) \, v(z), \quad 1-|z|\leq 2^{-N-2},
\end{equation}
with absolute comparison constants.

Let $\mathcal{D}(N)$ be the family of those (non-dilated) dyadic sectors $Q(I)$ of $\D$ for which
$|I| =  2^{-N-2}$ and $I \subset 2I_0$. If we denote $R_N = 1 - 2^{-N-2}$, then the Poisson integral formula yields
\begin{equation*}
  v(z_{Q_0}) 
  \gtrsim \frac{1}{|I_0|} \sum_{Q(I)\in \mathcal{D}(N)}   \int_{  e^{i t } \in I} v(R_Ne^{it}) \, dt
\end{equation*}
with an absolute comparison constant. Now, Harnack's inequality gives
\begin{equation} \label{eq:(6)}
  v(z_{Q_0})  
  \gtrsim \frac{1}{|I_0|} \sum_{Q\in \mathcal{D}(N)} v(z_Q) \, 2^{-N}
\end{equation}
with an absolute comparison constant.
Since $\Lambda\subset Q_0$, we have
\begin{equation*}
  1-|\overline{z}_n z_{Q_0}| \simeq |I_0|, \quad \big| 1-\overline{z}_n z_{Q_0} \big| 
  \simeq |I_0|, \quad z_n\in \Lambda,
\end{equation*}
and consequently,
\begin{equation*}
  v(z_{Q_0})  = \sum_{z_n\in\Lambda} \frac{1-|\overline{z}_n z_{Q_0}|^2}{|1-\overline{z}_n z_{Q_0}|^2}
  \simeq \frac{1}{|I_0|} \, \# \Lambda
\end{equation*}
with absolute comparison constants. Let us now estimate $ \# \Lambda$. We have
\begin{align*}
  \int_{e^{i\theta} \in I_0} G(e^{i\theta}) \, d\theta
  & = \int_{e^{i\theta}\in I_0} \left( \, \sum_{z_n\in \Lambda} 
    \frac{\ch_{\Gamma_\beta(e^{i\theta})}(z_n)}{1-|z_n|} \right) d\theta\\
  &  = \sum_{z_n\in \Lambda} \frac{\big| \big\{ e^{i\theta}\in I_0  : z_n \in \Gamma_\beta(e^{i\theta})\big\} \big|}{1-|z_n|}
    \simeq \# \Lambda
\end{align*}
with comparison constants depending on $\alpha$. By using \eqref{eq:(3)}, we conclude that 
$\#\Lambda \lesssim 2^N |I_0|$ which implies 
$v(z_{Q_0}) \lesssim 2^N$, and by \eqref{eq:(6)},
\begin{equation} \label{eq:(7)}
  \frac{1}{|I_0|} \sum_{Q\in \mathcal{D}(N)} v(z_Q) \, 2^{-N} 
  \lesssim 2^N
\end{equation}
with a comparison constant depending on $\alpha$.

For $j\in\N$, let $\mathcal{B}(j)$ be the collection of those sectors $Q\in \mathcal{D}(N)$ for which
$2^{j+N} \leq v(z_Q) < 2^{j+N+1}$. Moreover, let the collection $\mathcal{B}(0)$ consist of
sectors $Q\in \mathcal{D}(N) \setminus \bigcup_{j=1}^\infty \mathcal{B}(j)$. The estimate \eqref{eq:(7)} gives
\begin{equation} \label{eq:(81)}
  \# \mathcal{B}(j) \lesssim 2^{N-j} |I_0|, \quad j\in \N,
\end{equation}
with a comparison constant depending on $\alpha$. Since $1-|z_n| \geq 2^{-N-1}$ for any 
$z_n\in\Lambda$, it can be shown that $v(z) \simeq v(z_Q)$ 
whenever $z\in Q \in \mathcal{D}(N)$, with absolute comparison constants. Hence 
$v(z) \simeq 2^{j+N}$ if $z\in Q \in \mathcal{B}(j)$
for $j\in\N$, and $v(z) \lesssim 2^{N}$ if $z\in Q \in \mathcal{B}(0)$.
By \eqref{eq:(5)}, we deduce
\begin{equation} \label{eq:(8)}
  \log\frac{1}{|b(z)|} \simeq (1-|z|) \, 2^{j+N}, \quad z\in Q \in \mathcal{B}(j), \quad j \geq 1,
\end{equation}
and
\begin{equation} \label{eq:(9)}
  \log\frac{1}{|b(z)|} \lesssim (1-|z|) \, 2^{N}, \quad z\in Q \in \mathcal{B}(0), 
\end{equation}
with absolute comparison constants. For any $Q\in\mathcal{B}(j)$, $j\geq 0$, we denote 
\begin{equation*}
  F(Q)= \sum_{\substack{\widetilde{Q}\in\mathcal{A}\\z_{\widetilde{Q}}\in Q}} \ell(\widetilde{Q})^{1-p} . 
\end{equation*}
For any $Q\in \mathcal{B}(j)$ for $j\geq 1$, by \eqref{eq:(8)}, the sectors $\widetilde{Q}\in \mathcal{A}$
with $z_{\widetilde{Q}}\in Q$ must satisfy $\ell(\widetilde{Q})  \simeq  1-|z_{\widetilde{Q}}| \gtrsim 2^{-j-N}$,
with comparison constants depending on $\alpha$. We conclude
\begin{equation*}
  F(Q) \lesssim  \sum_{t=0}^j \, 2^t \left( 2^{-N-2-t} \right)^{1-p}
  \lesssim 2^{-N(1-p)} \, 2^{jp}, \quad Q\in \mathcal{B}(j), \quad j\geq 1,
\end{equation*}
with a comparison constant depending on $\alpha$ and $p$. By \eqref{eq:(81)}, 
\begin{equation*}
  \sum_{Q\in \mathcal{B}(j)} \, F(Q) 
  \lesssim 2^{-N(1-p)} \, 2^{jp} \cdot \# \mathcal{B}(j) 
  \lesssim 2^{Np} \, 2^{-j(1-p)} \, |I_0|, \quad j\geq 1,
\end{equation*}
and hence
\begin{equation} \label{eq:primertros}
  \sum_{j=1}^\infty \sum_{Q\in \mathcal{B}(j)} \, F(Q) 
  \lesssim 2^{Np} |I_0|
\end{equation}
with a comparison constant depending on $\alpha$ and $p$. For any $Q\in \mathcal{B}(0)$, by \eqref{eq:(9)}, the sectors $\widetilde{Q}\in \mathcal{A}$ with $z_{\widetilde{Q}}\in Q$ must satisfy $\ell(\widetilde{Q})  \gtrsim 2^{-N}$. Hence 
\begin{equation} \label{eq:segontros}
  \sum_{Q\in \mathcal{B}(0)} \, F(Q) 
  \lesssim 2^{Np} |I_0|  
\end{equation}
with a comparison constant depending on $\alpha$ and $p$.
Let $\mathcal{C}$ be the family of dilated dyadic sectors $Q \subset Q_0$ with $1-|z_Q| \geq 2^{-N-2}$. Then
\begin{equation} \label{eq:tercertros}
  \sum_{Q \in \mathcal{C}} \ell(Q)^{1-p}   \lesssim 2^{Np} |I_0|
\end{equation}
with a comparison constant depending on $\alpha$ and $p$. Now,  \eqref{eq:primertros}, \eqref{eq:segontros} and \eqref{eq:tercertros} give \eqref{eq:(4)}.


\subsubsection*{Step 3.}

Consider the (finite) Blaschke product
\begin{equation*}
  b^\star(z) = \left( \frac{z_{Q_0} - z}{1-\overline{z_{Q_0}} z} \right)^m \, 
  \prod_{Q\in\mathcal{A}} \frac{z_Q-z}{1-\overline{z_Q} z}, \quad z\in\D,
\end{equation*}
where $m= \max \,\{ n\in\N \cup \{0\} : n \leq 2^N |I_0| \}$. Estimate~\eqref{eq:(4)} gives
\begin{equation} \label{eq:zeroscount}
\begin{split}
  m^p (1-|z_{Q_0}|)^{1-p} + \sum_{Q\in\mathcal{A}} (1-|z_Q|)^{1-p}
  & \leq \left( 2^{N} |I_0| \right)^p  |I_0|^{1-p} + \sum_{Q\in\mathcal{A}} \ell(Q)^{1-p}\\
  & \lesssim  2^{Np} |I_0|
\end{split}
\end{equation}
with a comparison constant depending on $\alpha$ and $p$.

Let $\mathcal{Z}$ denote the zero-sequence of $b^\star$, where
zeros are counted according to their multiplicities. Note that $\mathcal{Z} \subset Q_0$.
Let $Q$ be a dyadic subsector of $Q_0$, and assume that there is a point $\zeta\in \mathcal{Z} \cap T(Q)$.
Since there exists an~absolute constant $0<\nu<1$ such that $\varrho(\zeta, z_Q)<\nu$, we deduce
\begin{equation*} 
  \frac{1-|\zeta|^2}{|1-\overline{\zeta} z|^2} \simeq  \frac{1-|z_Q|^2}{|1-\overline{z_Q} z|^2}, \quad z\in\D,
\end{equation*}
with absolute comparison constants.
Let $\mathcal{E}(Q_0)$ be the collection of dyadic sectors $Q\subset Q_0$.
Hence, by \eqref{eq:bderind}, 
\begin{equation*}
  | {b^\star}'(z) | 
  \lesssim 
  \sum_{Q\in \mathcal{E}(Q_0)} \!\!
  \# \big( \mathcal{Z} \cap T(Q) \big) \, \frac{1-|z_Q|^2}{|1-\overline{z_Q} z|^2},
  \quad z\in\D,
\end{equation*}
with an absolute comparison constant.
Since $1/2<p<1$, we deduce
\begin{align*}
  & \frac{1}{2\pi} \int_0^{2\pi} |{b^\star}'(re^{i\theta})|^p \, d\theta\\
  & \qquad \lesssim \sum_{Q\in \mathcal{E}(Q_0)} 
  \#\big( \mathcal{Z}  \cap T(Q) \big)^p \, (1-|z_Q|)^{1-p},  \quad 1/2 \leq |z_{Q_0}|\, r <1,
\end{align*}
with a comparison constant depending on $p$. Finally, according to \eqref{eq:zeroscount}, we have
$\bnm{{b^\star}'}_{H^p}^p \lesssim 2^{Np} |I_0|$ with a comparison constant depending on $\alpha$ and $p$.

We proceed to prove that $\nm{b'}_{L^p(\partial\D)}^p \lesssim \nm{{b^\star}'}_{L^p(\partial\D)}^p$.
By construction, there exist a constant $\sigma=\sigma(\alpha)$ with $0<\sigma<1$
satisfying the following property: if $|b^\star(z)|\geq \sigma$ for some $z\in\D$ then
$|b(z)|\geq 1/2$. Assume on contrary that $|b(z)|<1/2$ for some $z\in Q_0$.
Then, $|b(z_Q)|<\kappa$ where $Q$ is the dyadic subsector of $Q_0$ with $z\in T(Q)$.
Hence, $b^\star(z_Q)=0$ and 
$|b^\star(z)|\leq \varrho(z,z_Q)$, which proves the property. By Lemma~\ref{lemma:ccc},
we deduce $\nm{b'}_{L^p(\partial\D)}^p \lesssim \nm{{b^\star}'}_{L^p(\partial\D)}^p$,
with a comparison constant depending on $\alpha$ and $p$.
This completes the proof of \eqref{eq:(2)}, and hence the proof of Theorem~\ref{thm:main}.


\section{Proof of the corollaries} \label{sec:prfcor}


\begin{proof}[Proof of Corollary~\ref{cor:cohnimp}]
If $B$ is a Blaschke product whose zeros satisfy \eqref{eq:protas}, then $B'\in H^p$ by \cite[Theorem~2]{P:1973}.
Note that no assumption on the separation of zeros is needed here. 
Conversely, suppose that $B$ is a Blaschke product whose zeros $\{ z_n \}$
are separated, and $B' \in H^p$ for some $1/2<p<1$. Fix $1<\alpha<\infty$. By \eqref{eq:Bprimerep},
\begin{equation*}
  \big| B'(e^{i\theta}) \big| 
  \geq \frac{1}{\alpha^2} \, \sum_{z_n\in \:\!\Gamma_\alpha(e^{i\theta})}  \frac{1}{1-|z_n|}, 
  \quad \text{ a.e. } e^{i\theta}\in\partial\D.
\end{equation*}
Define $F^{\star} : \partial\D \to \R \cup \{\infty\}$ by
\begin{equation*}
  F^\star(e^{i\theta}) = \begin{cases}
    \left( \, \inf\limits_{z_n\in \:\! \Gamma_\alpha(e^{i\theta})} \big|e^{i\theta}-z_n \big| \right)^{-1}, 
    & \text{if $\{z_n\} \cap \Gamma_\alpha(e^{i\theta}) \neq \emptyset$,}\\
    0, & \text{otherwise.}
  \end{cases}
\end{equation*}
Note that $F^\star$ is finite almost everywhere, since \eqref{eq:lemmaa} ensures that 
the intersection $\{z_n\} \, \cap \, \Gamma_\alpha(e^{i\theta})$ contains at most finitely many
points for a.e.~$e^{i\theta}\in\partial\D$. It is also clear that $F^* (e^{i\theta}) \leq {\alpha}^2 \, |B'(e^{i\theta})|$ 
for a.e.~$e^{i\theta} \in \partial \D$. The assumption $B'\in H^p$ implies that $F^\star\in L^p(\partial\D)$. 

For $j\geq 1$, consider the annulus $R_j =\{z \in \D : 2^{-j-1} \leq 1-|z| < 2^{-j}\}$. 
Define $J(e^{i\theta}) \in \N \cup \{\infty\}$ for those $e^{i\theta}\in\partial\D$ 
for which $\{z_n\} \cap \Gamma_\alpha(e^{i\theta}) \neq \emptyset$ by the following rules.
If the infimum
\begin{equation} \label{eq:inf}
  \inf_{z_n\in \:\!\Gamma_\alpha(e^{i\theta})} \big|e^{i\theta}-z_n \big|
\end{equation}
is strictly positive, then let $J(e^{i\theta})\in\N$ be the index such that the annulus $R_{J(e^{i\theta})}$ contains the point
at which the infimum \eqref{eq:inf} is attained. If the infimum \eqref{eq:inf} is zero, then define $J(e^{i\theta}) = \infty$.
Note that $J(e^{i\theta})$ can be infinite only on a set of Lebesgue measure zero.

Since $\{z_n \}$ is separated, there is a constant $S=S(\alpha, \{z_n\})$
such that 
\begin{equation*} 
  \# \big( \{z_n \} \cap R_j \cap \Gamma_\alpha(e^{i\theta}) \big) \leq S< \infty,
  \quad e^{i\theta}\in\partial\D, \quad j\in\N.
\end{equation*}
For those $e^{i\theta}\in\partial\D$ for which $\{z_n\} \cap \Gamma_\alpha(e^{i\theta}) \neq \emptyset$
and $J(e^{i\theta})<\infty$, we obtain
\begin{equation*} 
  \sum_{z_n\in \:\!\Gamma_\alpha(e^{i\theta})}  \frac{1}{(1-|z_n|)^p}
  =  \sum_{j=1}^{J(e^{i\theta})} \sum_{z_n\in \:\!\Gamma_\alpha(e^{i\theta}) \cap R_j}  \frac{1}{(1-|z_n|)^p}
  \leq S  \sum_{j=1}^{J(e^{i\theta})} 2^{j p} 
  \lesssim F^\star(e^{i\theta})^p,
\end{equation*}
while otherwise this estimate is trivial. Now
\begin{equation*}
  \int_0^{2\pi} F^\star(e^{i\theta})^p \, d\theta
  \, \gtrsim \, \int_0^{2\pi} \Bigg( \sum_n \frac{\ch_{\Gamma_\alpha(e^{i\theta})}(z_n)}{(1-|z_n|)^p}  \Bigg) d\theta
  \, \simeq \, \sum_n (1-|z_n|)^{1-p}.
\end{equation*}
This completes the proof of Corollary~\ref{cor:cohnimp}.
\end{proof}

 
\begin{proof}[Proof of Corollary~\ref{cor:m1}]
Without loss of generality, we may assume  that  $\{z_n\}$ is contained in
$\Gamma_\beta(1)$. Let $\alpha=2\beta$. By elementary geometry,
there exist natural numbers $k_1=k_1(\beta)$ and $k_2=k_2(\beta)$  such that:
\begin{enumerate}
\item[\rm (i)]
if $z_n \in \Gamma_\alpha(e^{i\theta})$ and $2^{-k+1} < | e^{i\theta} - 1| \leq 2^{-k+2}$ for some $k\in\N$,
then $1-|z_n| > 2^{-k-k_1}$;

\item[\rm(ii)]
if $1-|z_n| > 2^{-k+k_2}$ and $2^{-k+1} < | e^{i\theta} - 1| \leq 2^{-k+2}$ for some $k>k_2$,
then $z_n \in \Gamma_\alpha(e^{i\theta})$.
\end{enumerate}
For $k \in \N$, consider $I_k = \{ e^{i\theta} \in\, \partial\D \, :\,   2^{-k+1} < | e^{i\theta} - 1| \leq 2^{-k+2} \}$. 
Let $1/2<p<1$, and consider the function $F=F_{\alpha,B}$ in \eqref{eq:defF}. We deduce
\begin{equation*}
  F(e^{i\theta}) 
  \leq \sum_{j=1}^{k+k_1} \sum_{z_n \in R_j} \frac{1}{1-|z_n|} \leq \sum_{j=1}^{k+k_1} 2^j N_j,
  \quad  e^{i\theta} \in I_k, \quad k\in\N.
\end{equation*}
Since $1/2<p<1$, we have
\begin{align*}
  \int_0^{2\pi} F(e^{i\theta})^p \, d\theta
  & = \sum_{k=1}^\infty \, \int_{I_k} F(e^{i\theta})^p \, d\theta\\
  & \lesssim \sum_{k=1}^\infty 2^{-k} \left( \, \sum_{j=1}^{k+k_1} 2^j N_j \right)^p
    \leq \sum_{k=1}^\infty 2^{-k} \left( \, \sum_{j=1}^{k+k_1} 2^{jp} N_j^p \right) \\
  & = \sum_{j=1}^\infty 2^{jp} N_j^p \left( \, \sum_{k=j-k_1}^{\infty} 2^{-k}  \right)
    \lesssim \sum_{j=1}^\infty 2^{-j(1-p)} N_j^p
\end{align*}
with a comparison constant depending on $\beta$. Similarly,
\begin{equation*}
  F(e^{i\theta}) 
  \geq \sum_{j=1}^{k-k_2} \sum_{z_n \in R_j} \frac{1}{1-|z_n|} 
  \geq \frac{1}{2} \, \sum_{j=1}^{k-k_2} 2^j N_j,
  \quad e^{i\theta} \in I_k, \quad k> k_2,
\end{equation*}
and hence
\begin{align*}
  \int_0^{2\pi} F(e^{i\theta})^p \, d\theta
  & \geq  \sum_{k=k_2+1}^\infty \, \int_{I_k} F(e^{i\theta})^p \, d\theta\\
  & \gtrsim \sum_{k=k_2+1}^\infty 2^{-k} \left( \, \sum_{j=1}^{k-k_2} 2^j N_j \right)^p
    \geq \sum_{k=k_2+1}^\infty 2^{-k} \cdot 2^{(k-k_2)p} N_{k-k_2}^p \\
  & \gtrsim \sum_{k=1}^\infty 2^{-k(1-p)} N_k^p 
\end{align*}
with a comparison constant depending on $\beta$ and $p$. We have proved that
\begin{equation*}
  \nm{F}_{L^p(\partial\D)}^p \, \simeq \,\sum_{k=1}^\infty 2^{-k(1-p)} N_k^p 
\end{equation*}
with comparison constants depending on $\beta$ and $p$. The assertion follows by applying Theorem~\ref{thm:main}.
\end{proof}


\begin{proof}[Proof of Corollary~\ref{cor:m2}] Given $a \in \D$, let $\tau_a $ be 
the automorphism of the unit disc given by $\tau_a (w) = (a-w)/(1- \overline{a}w)$, $w \in \D$. 
If $\tau_a \circ B$ is a Blaschke product and  \eqref{eq:preimages} holds, by \eqref{eq:protas}, 
we deduce $(\tau_a \circ B)' \in H^p$ and hence $B' \in H^p$. 
Conversely, assume that $B$ is a Blaschke product with $B' \in H^p$ for $1/2<p<1$, and let us show the last part of the statement. 
By the result of Carleson mentioned in the Introduction, it is sufficient to show that for all $a \in \D$, except possibly 
for a set of logarithmic capacity zero, we have
\begin{equation*}
  F(a) = \int_{\D} \frac{\log |(\tau_a \circ B) (z)|^{-1}}{(1-|z|)^{1+p}} \, dm(z) < \infty. 
\end{equation*}
Fix $0<c<1$. Given a probability measure $\sigma$, which is supported in the disc of radius~$c$ 
centered at the origin, and which satisfies
\begin{equation} \label{eq:potential}
  M= \sup_{z \in \D} \int_{\D} \log |\tau_a (z) |^{-1} \, d \sigma (a) < \infty, 
\end{equation}
we have to show
  $\sigma \big(\{a \in \D : F(a) = \infty \}\big)=0$.  
Let  $E=\{z \in \D : |B(z)| \leq (1+c)/2 \}$. Fubini's theorem gives
\begin{equation*}
  \int_{\D} F(a) \, d\sigma(a) = {\rm(I)} + {\rm(II)}, 
\end{equation*}
where
\begin{align*}
  {\rm (I)} &= \int_{E}\int_{\D} \frac{\log |(\tau_a \circ B) (z)|^{-1} }{(1-|z|)^{1+p}} \, d \sigma (a) dm(z) , \\ 
  {\rm (II)} &= \int_{\D \setminus E}\int_{\D} \frac{\log |(\tau_a \circ B) (z)|^{-1}}{(1-|z|)^{1+p}} \, d \sigma (a) dm(z) . 
\end{align*}
Using \eqref{eq:potential} and part (b) of  Theorem~\ref{thm:main}, we obtain 
\begin{equation*}
  {\rm (I)} \leq  M \int_{E} \frac{dm(z)}{(1-|z|)^{1+p}}  < \infty . 
\end{equation*}
Observe that $|B(z)| > (1+c)/2$ for any $z \in \D \setminus E$. Since $\varrho(c, (1+c)/2)>0$, 
there exists a constant $C_1=C_1(c) >0$ such that
\begin{equation*}
  \log |(\tau_a \circ B) (z)|^{-1} \leq C_1 (1-|B(z)|) , \quad z \in \D \setminus E , \quad |a|\leq c. 
\end{equation*}
Since $\sigma$ is supported in $\{z \in \D : |z| \leq c \}$, we deduce
\begin{equation*}
  {\rm(II)} \leq C_1 \int_{\D} \frac{1-|B(z)|}{(1-|z|)^{1+p}} \, dm(z). 
\end{equation*}
Because $1-|B(r e^{i \theta})| \leq \int_r^1 |B' (s e^{i \theta})| \, ds$ for a.e.~$e^{i \theta} \in \partial \D$, 
by using polar coordinates and Fubini's theorem, we deduce
\begin{equation*}
  {\rm(II)} \leq \frac{C_1}{p} \int_{\D} \frac{|B'(z)|}{(1-|z|)^{p}} \, dm(z) < \infty, 
\end{equation*}
 by \eqref{eq:sobolev} with $q=1$ and $s=p$. This finishes the proof.   
\end{proof}


\begin{proof}[Proof of Corollary~\ref{thm:F}]
Let us first prove the necessity of the condition \eqref{eq:F}. 
Fix $\alpha >1 $ sufficiently large such that for any $Q(I)\in \mathcal{E}$ 
and any $e^{i \theta} \in I$, we have $T(R) \subset \Gamma_{\alpha} (e^{i \theta})$
for all $R\in\mathcal{E}$ for which $\ell(R)\geq \ell(Q(I))$ and $Q(I) \subset 3R$. 
Consider the function $F= F_{\alpha, B}$ defined in~\eqref{eq:defF}. 
Observe that if $Q(I) \in {\mathcal{E}}_N$ then $I \subset \{e^{i \theta} \in \partial \D : F(e^{i \theta}) > 2^N\}$. 
Hence $\ell ({\mathcal{E}}_N) \leq |\{e^{i \theta}  \in \partial \D : F(e^{i \theta}) > 2^N\}|$ 
and the sum in \eqref{eq:F} is bounded by a~fixed multiple of $\|F\|_{L^p (\partial \D)}^p$. Since $B' \in H^p$,  
the condition \eqref{eq:F} is satisfied by part (c) of Theorem~\ref{thm:main}. 

Conversely, assume that \eqref{eq:F} holds and let us show that $B' \in H^p$. Pick $\alpha > 1$ 
close enough to $1$ such that the following conditions are satisfied:
\begin{itemize}
\item[\rm (i)]
for any sector $Q(I)$ and any point $e^{i \theta} \in \partial \D$ 
with $T(Q) \cap \Gamma_{\alpha} (e^{i \theta}) \neq \emptyset $, we have $e^{i \theta} \in 2I$;
\item[\rm (ii)]
each Stolz angle $\Gamma_{\alpha}(e^{i\theta})$, for $e^{i\theta}\in\partial\D$, intersects
at most two top parts of sectors, which are of the same size and belong to $\mathcal{E}$.
\end{itemize}
We proceed to prove
\begin{equation} \label{eq:contingut}
  \big\{e^{i \theta} \in \partial \D : F(e^{i \theta}) >  2^{N+1} \big\} \subset \bigcup_{Q(I) \in {\mathcal{E}}_N} 2I, \quad N \in \N .  
\end{equation}
Actually, assume that $e^{i \theta} \notin 2 I$ for any $Q=Q(I) \in {\mathcal{E}}_N$. Consequently, we have
$T(Q) \cap \Gamma_{\alpha} (e^{i \theta}) = \emptyset$ for any $Q \in {\mathcal{E}}_N$. 
For $n \in \N$, let $Q_n(e^{i\theta})\in\mathcal{E}$ be the sector, 
whose normalized length is $2^{-n}$ and whose base contains the point $e^{i\theta}\in\D$.
Since
\begin{equation*}
  \sum_{\substack{z_k\in\Gamma_\alpha(e^{i\theta})\\1-|z_k| \geq 2^{-n-1}}} \frac{1}{1-|z_k|} 
  \leq 2 \cdot F\big( Q_n(e^{i\theta})\big) \leq 2^{N+1},
\quad n\in\N,
\end{equation*}
we deduce that $F(e^{i \theta}) \leq 2^{N+1}$. This 
proves \eqref{eq:contingut}. Now \eqref{eq:F} and \eqref{eq:contingut} give $F \in L^p (\partial \D)$,
and by Theorem~\ref{thm:main}, we conclude $B' \in H^p$. 
\end{proof}


\begin{proof}[Proof of Corollary~\ref{thm:protasimp}]
Let $Q$ be a dyadic sector of $\D$ for which $0<\ell(Q)<1/2$, and assume that $z_n\in T(Q)$. 
There exists an~absolute constant $0<\sigma<1$
such that $\varrho(z_n,z_Q)<\sigma$, and hence
\begin{equation*} 
  \frac{1-|z_n|^2}{|1-\overline{z}_n z|^2} \simeq  \frac{1-|z_Q|^2}{|1-\overline{z_Q} z|^2}, \quad z\in\D,
\end{equation*}
with absolute comparison constants. By applying \eqref{eq:bderind}, we obtain
\begin{equation*}
  \big| B'(z) \big| 
  \lesssim \sum_{Q \in \mathcal{E}} N(Q) \, \frac{1-|z_Q|^2}{|1-\overline{z_Q} z|^2},
  \quad z\in\D.
\end{equation*}
Since $1/2<p<1$, we conclude
\begin{equation*}
  \int_0^{2 \pi}|B'(re^{i\theta})|^p \, d\theta
  \lesssim \sum_{Q \in \mathcal{E}} N(Q)^p \int_0^{2\pi}  \frac{(1-|z_Q|^2)^p}{|1-\overline{z_Q} r e^{i\theta}|^{2p}} \, d\theta
\end{equation*}
for any $0<r<1$. Since the last integral is comparable to $\ell (Q)^{1-p}$, the assertion in (a) follows.

The proof of (b) is more involved and is based on Theorem \ref{thm:main}. Since $0<\varepsilon < 1$ is fixed 
we write $\mathcal{A}(Q) = \mathcal{A}_{\varepsilon} (Q)$.  Let $\mathcal{E}_1$ be the family of maximal dyadic 
sectors $Q \in \mathcal{E}$ such that $N(Q) \geq C$. The family $\mathcal{E}_1$ will be called the first generation. 
For each $Q \in \mathcal{E}_1$, consider only the zeros which are not contained in the top parts of dyadic sectors 
of the family $\mathcal{A}(Q)$. Then the second generation $\mathcal{E}_2$ is defined as the maximal dyadic 
sectors $Q \in \mathcal{E}$ strictly contained in sectors of the first generation such that $N(Q) \geq C$. 
Observe that 
we do not count the zeros located in top parts of sectors of the family $\mathcal{A} (Q)$ for $Q \in \mathcal{E}_1$. 
For each sector $Q \in \mathcal{E}_2$, consider only the zeros which are not contained in top parts of dyadic sectors 
of the family $\mathcal{A}(Q)$, and continue the construction inductively.

On one hand, the subsequence of zeros $\{z_n \}$ contained in top parts of dyadic sectors $Q \in \mathcal{E}$ 
with $N(Q) < C$ is a finite union of separated sequences. Hence, by Corollary \ref{cor:cohnimp},
\begin{equation} \label{light}
  \sum_{\substack{Q \in \mathcal{E}\\ N(Q) < C}} N(Q)^{p} \ell(Q)^{1-p} < \infty .
\end{equation}
On the other hand, dyadic sectors $Q$ such that $N(Q) \geq C$ are either in the family $ \mathcal{E}_j $ 
for some $j\in\N$ or in $\mathcal{A} (Q)$ for some $Q \in  \mathcal{E}_j$. Hence
\begin{equation} \label{heavy}
\begin{split}
  & \sum_{\substack{Q \in \mathcal{E}\\ N(Q) \geq  C}} N(Q)^{p} \ell(Q)^{1-p}\\
  & \qquad \leq \sum_{j=1}^\infty  \left( \, \sum_{Q \in \mathcal{E}_j} \Bigg(   N(Q)^{p} \ell(Q)^{1-p} 
    + \sum_{R \in \mathcal{A}(Q)} N(R)^{p} \ell(R)^{1-p} \Bigg) \right).
\end{split} 
\end{equation}

Observe that for any $Q \in \mathcal{E}_j$ and any $R \in \mathcal{A}(Q)$, we have 
$N(R)/ \ell(R) < (1+ \varepsilon) N(Q)/ \ell(Q)$. Hence
\begin{equation*}
  \sum_{R \in \mathcal{A}(Q)} N(R)^{p} \ell(R)^{1-p} \leq (1+ \varepsilon)^p \, \frac{N(Q)^p}{\ell(Q)^p} \sum_{R \in \mathcal{A}(Q)} \ell(R). 
\end{equation*}
Using the assumption, we deduce
\begin{equation} \label{heavy1}
  \sum_{R \in \mathcal{A}(Q)} N(R)^{p} \ell(R)^{1-p} 
  \leq C (1+ \varepsilon)^p N(Q)^{p} \ell(Q)^{1-p} . 
\end{equation}
Hence, by  \eqref{light}, \eqref{heavy} and \eqref{heavy1}, the proof reduces to the estimate 
\begin{equation} \label{superheavy}
  \sum_{j=1}^\infty \sum_{Q \in \mathcal{E}_j}   N(Q)^{p} \ell(Q)^{1-p} < \infty . 
\end{equation}

Fix $Q \in \mathcal{E}_j $ for $j \in\N$, and consider the set 
\begin{equation*}
  L(Q)=\left\{z\in Q :  (1- \varepsilon) \, \frac{\ell(Q)}{N(Q)} < 1-|z| < (1 + \varepsilon) \, \frac{\ell(Q)}{N(Q)} \right\}.
\end{equation*}
Since the Blaschke product $B$ has $N(Q)$ zeros in $T(Q)$, there exists a universal constant $0<c<1$ such that
$|B(z)|< c$ for  $z \in L(Q)$.
This follows from the estimates $|B(z)| \lesssim \varrho(z, z_Q)^{N(Q)} $ and $1-\varrho(z, z_Q)^2 \simeq N(Q)^{-1}$
for $z\in L(Q)$. By construction, the sets $\{L(Q): Q \in \bigcup_{j \geq 0} \mathcal{E}_j \}$ are pairwise disjoint. Hence 
\begin{equation*}
  \sum_{j=1}^\infty \sum_{Q \in \mathcal{E}_j} \, \int_{L(Q)} \frac{dm(z)}{(1-|z|)^{1+p}} 
  \leq \int_{\{z\in\D \, : \, |B(z)|<c \}} \frac{dm(z)}{(1-|z|)^{1+p}} < \infty.
\end{equation*}
Since 
\begin{equation*}
  \int_{L(Q)} \frac{dm(z)}{(1-|z|)^{1+p}} \simeq N(Q)^p  \ell(Q)^{1-p}, 
\end{equation*}
estimate \eqref{superheavy} follows and the proof is completed. 
\end{proof}


\section{Questions} \label{sec:questions}

We conclude this paper with two questions.
First, consider those Blaschke products whose derivative belongs to the weak Hardy space
$H^p_w$ for $1/2<p<1$.
These are precisely the Blaschke products~$B$, for which
there exists a constant $C=C(B)$ with $0<C<\infty$, such that
\begin{equation*}
\big| \big\{ e^{i\theta} \in \partial\D : |B'(re^{i\theta})| > \lambda\big\} \big| \leq \frac{C}{\lambda^p},
\quad 0<\lambda<\infty,
\end{equation*}
for any $0<r<1$. 
Blaschke products whose derivative belongs to
$H^1_w$, and Blaschke products in the weak Besov spaces, have been described in \cite{CN:2015,GN:2014}.
Is it true that $B'\in H^p_w$ if and only if $F_{\alpha,B}\in L^p_w(\partial\D)$?

Second, consider Blaschke products $B$ that satisfy the Besov-type condition
\begin{equation*}
\int_0^1 M_p(r,B')^q \, (1-r)^{(1-\alpha)q-1} \, dr < \infty,
\end{equation*}
where $0<\alpha<1$, $0<p\leq \infty$ and $0<q<\infty$. As usual, $M_p(r,B')$ denotes the integral
mean of $B'$. For a general reference concerning inner functions in the Besov-type spaces, see~\cite{V:1984}.
Is it possible to describe these Blaschke products (for a certain range of parameters)
in terms of the geometric zero distribution, in particular, by using an auxiliary function
reminiscent of $F_{\alpha,B}$?

\bigskip
\noindent
\textbf{Acknowledgements.}
The authors would like to thank J.~R\"atty\"a for pointing out the reference \cite{L:1991}.


\end{document}